\documentclass[reqno,11pt]{amsart}
\usepackage{a4wide}
\usepackage{amsmath}
\usepackage{amsthm}
\usepackage{amssymb}
\usepackage{amsfonts}
\usepackage{color}
\usepackage{enumerate}
\usepackage{graphicx}

\newtheorem{Th}{Theorem}

\newtheorem{cor}{Corolary}
\newtheorem{prop}{Proposition}

\newtheorem{obs}{Remark}
\newtheorem{lema}{Lemma}
\def\rr{\mathbb{R}}

\def\nn{\mathbb{N}}
\def\cn{(N-2)^2/4}
\def\into{\int_{\Omega}}
\def\hoi{H_{0}^{1}(\Omega)}

\def\intR{\int_{0}^{R}}

\def\dtheta{\mathrm{d\theta}}
\def\pa{\partial}

\def\con{\mathcal{C}_\gamma}
\def\d{\partial}

\def\eps{\varepsilon}
\def\ig{\int_{0}^{\gamma}}
\def\igo{\int_{0}^{\overline{\gamma}}}
\def\eps{\varepsilon}

\def\nn{\mathbb{N}}

\def\cn{(N-2)^2/4}
\def\into{\int_{\Omega}}
\def\hoi{H_{0}^{1}(\Omega)}

\def\intR{\int_{0}^{R}}

\def\d{\partial}
\def\eps{\varepsilon}

\def\be{\begin{equation}}
\def\ee{\end{equation}}

\def\intog{\int_{0}^{\gamma}}

\def \hogt{H_{0}^{1}([0,\gamma],  t^{N-2} dt  )}
\def \loogt{L^{2}([0,\gamma],  t^{N-2} dt  )}
\def \loogdt{L^{2}([0,\gamma],  t^{N-3} dt  )}
\numberwithin{equation}{section} \numberwithin{Th}{section}
\numberwithin{cor}{section} \numberwithin{lema}{section}
\numberwithin{prop}{section} \numberwithin{obs}{section}
\numberwithin{Def}{section}

\title[Hardy inequalities with boundary singularities]{Hardy inequalities with boundary singularities}
\author{\bigskip Cristian Cazacu}
\thanks{\noindent \textit{ E-mail address: cazacu@bcamath.org}\\
Partially supported by the Grant MTM2008-03541 of the MICINN, Spain,
the ERC Advanced Grant FP7-246775 NUMERIWAVES, the Grant PI2010-04
of the Basque Government, and the doctoral fellowship FPU-UAM from
Universidad Aut\'{o}noma de Madrid.}
\begin{document}
\maketitle
\begin{center}
\small{BCAM-Basque Center for Applied Mathematics,
 Biskaia Technology Park, Building 500,\\ E-48160,
Derio-Basque Country-Spain

and

Departamento de Matem\'{a}ticas, Facultad de Ciencias, Universidad
Aut\'{o}noma de Madrid,\\
 28049 Madrid, Spain}
\end{center}

\begin{abstract}
In this work  we prove some Hardy-Poincar\'{e} inequalities
with quadratic singular potentials localized on the boundary of a
smooth domain.  Then, we consider conical domains with vertex on the
singularity and we show upper and lower bounds for the corresponding
optimal constants in the Hardy inequality. In particular, we prove
the asymptotic behavior of the optimal  constant when the slot of
the cone tends to zero.

\end{abstract}
\section{Introduction}
  Hardy inequalities represent a classical
subject in which there has been intensive research in the recent
past, mainly motivated by its applications to Partial Differential
Equations (PDE's) and, more precisely, with the positivity of the
Schr\"{o}dinger operator
\begin{equation}\label{eq165}
L_\lambda:=-\Delta -\frac{\lambda}{|x|^2}, \textrm{ with } \lambda
\in \rr,
\end{equation}
involving the inverse square singular potential $1/|x|^2$.

Inverse square potentials are interesting because of their
criticality since they are homogeneous of degree -2. They often
appear in the linearization of critical nonlinear PDE's playing a
crucial role in the asymptotic behaviour of branches of solutions in
bifurcation problems (e.g. \cite{MR1605678}, \cite{MR1616516},
\cite{MR1760280}). \noindent The operator $L_\lambda$ arises in
physics and in particular  in the relativity theory and quantum
mechanics \cite{simon}. We also mention other interesting
applications in molecular physics \cite{pot1}, quantum cosmology
\cite{pot2}, combustion models \cite{MR1616905}, brownian motion
\cite{David}, etc.

In \cite{hard}, G. H Hardy proved  that, in the one dimensional
case, the optimal inequality
\begin{equation}\label{eq166} \int_{0}^{\infty}|u'(r)|^2dr\geq
\frac{1}{4}\int_{0}^{\infty}\frac{u^2(r)}{r^2}dr,
\end{equation}
holds for
 functions belonging to $H_{0}^{1}(0,\infty)$. More precisely,
\begin{equation}\label{eq113}
\inf_{u\in H_{0}^{1}(0,\infty)}\Big(\int_{0}^{\infty}|u'(r)|^2dr
\Big/\int_{0}^{\infty}\frac{u^2}{r^2}dr \Big)=\frac{1}{4}.
\end{equation}
The classical multi-dimensional Hardy inequality (cf.
\cite{hardy-polya}) asserts that for any $\Omega$ an open subset of
$\rr^N$, $N\geq 3$,  it holds that
\begin{equation}{\label{eq167}}
\int_{\Omega}|\nabla u|^2dx\geq
\frac{(N-2)^2}{4}\int_{\Omega}\frac{u^2}{|x|^2}dx,
\end{equation}
for all $u\in H_{0}^{1}(\Omega)$. Moreover, if $\Omega$ contains the
origin, the constant $(N-2)^2/4$ is optimal and it is not attained.
For $N=2$, inequality \eqref{eq167} is trivially true.

The reader interested in the existing literature on the extensions
of the classical Hardy inequality  is referred, in particular, to
the following papers and the references therein: \cite{MR1616905},
\cite{adimurthi}, \cite{MR768824},
 \cite{barbatis},
\cite{filipas}, \cite{MR2099546},
\cite{tintarev}. Recently, improved versions of (\ref{eq167}) have
been established in open bounded domains containing the origin (see
\cite{MR1605678}, \cite{MR1760280}, \cite{adimurthi1}). We also
mention the papers \cite{MR2379440}, \cite{MR2209764} and the
references therein for discussing  inequalities with multipolar
singularities. There has been also an intensive research for
singular potentials involving the distance to the boundary (e.g.
\cite{marcus}, \cite{marcus2}).

\vspace{0.5cm}

 However, Hardy inequalities with one singular potential,  in which the singularity
 lies
 on the boundary
have been less investigated. This paper is mainly devoted to analyze
this issue. To be more precise, throughout the paper, we consider
$\Omega$ to be a subset of $\rr^N$ with the origin $x=0$ placed on
its boundary $\pa\Omega$, where the singularity is located. We then
define $\mu(\Omega)$ as the best constant
 in the inequality
\begin{equation}\label{eq143}
\int_{\Omega}|\nabla u|^2dx \geq
\mu(\Omega)\int_{\Omega}\frac{u^2}{|x|^2}dx, \quad \forall \quad
u\in \hoi,
\end{equation}
i.e.  $$\mu(\Omega):=\inf\Big\{\int_{\Omega}|\nabla u|^2dx\Big/
\int_{\Omega}u^2/|x|^2dx,\  u\in H_{0}^{1}(\Omega)\Big\}.$$ Of
course, in view of  ({eq167}),  $\mu(\Omega)\geq \cn$. The authors
in \cite{musina} showed that the strict inequality
$\mu(\Omega)>(N-2)^2/4$ holds true
 when $\Omega$ is a bounded domain of class $C^2$.  
Actually, the value  $\mu(\Omega)$ depends on the geometric
properties of the boundary $\pa \Omega$ at the singularity. The
first explicit case has been given for $\Omega=\rr_{+}^{N}$, where
$\rr_+^N$ is the half-space of $\rr^N$ in which the condition
$x_N>0$ holds. More precisely,  for any $N\geq 1$, Filippas,
Tertikas and Tidblom proved in \cite{terti} the new Hardy
inequality:
\begin{equation}\label{equ1}
\int_{\rr_+^N}|\nabla u|^2dx\geq
\frac{N^2}{4}\int_{\rr_+^N}\frac{u^2}{|x|^2}dx \quad \forall\quad
u\in H_{0}^{1}(\rr_+^N).
\end{equation}
Moreover, they proved the constant $N^2/4$ to be optimal (cf.
Corollary 2.4, pp. 12, \cite{terti}) i.e. $\mu(\rr_{+}^{N})=N^2/4$.

As a consequence of this result, mainly by the invariance under
dilatations, it holds that $\mu(\Omega)=N^2/4$ for any domain
$\Omega$ of class $C^2$ with the support in the half-space
$\rr_{+}^{N}$.
 Moreover, if $\Omega$ is a bounded domain contained in the
half-space, the following improved Hardy-Poincar\'{e} holds for any
$u\in \hoi$ (cf. \cite{musina}):
\begin{equation}\label{HPI}
\into |\nabla u|^2 dx -\frac{N^2}{4}\into \frac{u^2}{|x|^2}dx \geq
\frac{\lambda(\mathbb{D})}{|\textrm{diam}(\Omega)|^2}\into u^2 dx,
\end{equation}
where $\lambda_1(\Omega)$ is the first eigenvalue of the Laplacian
in the unit ball in 2-d.

Another interesting situation  appears in non-smooth domains
$\Omega$, when the boundary develops corners or cusps at the
singularity. The most relevant example of such a domain is
represented by a cone with the vertex at the origin $x=0$. The
question of studying the exact value of $\mu(\Omega)$ in cones has
been full-filled in 2-d case. More precisely, if
$\mathcal{C}_{\gamma}$ is the conical sector with the slot
$\gamma\in (0, 2\pi)$, then (cf. \cite{caldiroli1})
$$\mu(\mathcal{C}_{\gamma})=\frac{\pi^2}{\gamma^2}.$$
By our knowledge, for higher dimensions $N\geq 3$, the value
$\mu(\Omega)$ is only known when the cone $\Omega$ coincides with
the half-space $\rr_{+}^{N}$.

In this paper, the aim we focus on is two folded.

  Roughly
speaking, in the first part, we improve some results in
\cite{musina} addressing Hardy-Poincar\'{e} inequalities in smooth
domains. Besides, we complete rigorous proofs of several results
stated in \cite{musina}. In the second part, the new issue addressed
is related to optimal Hardy inequalities in conical domains in
dimensions $N\geq 3$.

Part of the results in this paper were first announced in
\cite{cristi}. Soon after that, the preprint \cite{musina} has been
submitted for publication while this extended version of the paper
was being prepared. Because  \cite{musina} yields some similar
results, we thus present here in a detailed manner the most novel
aspects of the note \cite{cristi} not addressed in \cite{musina}.

  Let us now resume the content of the
paper and the main results we obtain. 
In Section \ref{Sec1} we show Hardy-Poincar\'{e} inequalities in
bounded smooth domains completing and extending some results in
\cite{musina}. We mainly refer  to Theorems \ref{tu2}, \ref{tu3},
\ref{tu7}. These results turn out to be closely related to the
ellipticity of $\Omega$ at the origin, but also to the global
geometry of $\Omega$. When $\Omega$ is not elliptic at the origin, a
weaker Hardy inequality holds (cf. Theorem \ref{ta5}) and  the
continuous dependence of the Hardy constant in cones is required in
the proof. This last result is rigorously shown in Section
\ref{Sec2}, Theorem \ref{taa3}. Besides,
  in Section \ref{Sec2} we prove lower and upper bounds for
the optimal constant $\mu(\Omega)$, when $\Omega$ is a cone in
dimensions $N\geq3$, with the vertex in $x=0$. In particular,  the
asymptotic value of $\mu(\Omega)$ is shown when the slot of the cone
tends to zero (see Remark \ref{asymptotic}). We point out that all
the sections contain at least a small introduction at the beginning.
In Section \ref{sec4} we conclude with some comments.

\section{Inequalities in smooth domains}\label{Sec1}

As we said above, the value of the optimal constant $\mu(\Omega)$
depends on the various geometric properties of $\Omega$. In this
section we assume $\Omega$ to be a Lipschitz domain with smooth
boundary around the origin.
  Then $\pa \Omega$ is an $(N-1)$-Riemannian submanifold of
$\rr^N$ and assume that $\alpha_1,\alpha_2,...,\alpha_{N-1}$ are the
principal curvatures of $\pa \Omega$ at 0. Then, up to a rotation
(cf. \cite{boothby}, \cite{gosu}), the boundary near the origin can
be written as
\begin{equation}\label{eq116}
x_N=h(x')=\sum_{i=1}^{N-1}\alpha_ix_{i}^{2}+o(|x'|^2) \textrm{ as }
|x'|\rightarrow 0,
\end{equation}
where $x'=(x_1,...,x_{N-1},0)$.  It is well-known that the principal
curvatures are the eigenvalues of the $2nd$ fundamental form of the
surface $\pa \Omega$ (cf. \cite{gallot}). If we choose
\begin{equation}\label{eq117}
\gamma <\min\{\alpha_i: 1\leq i\leq N\},
\end{equation}
then $x_N> \gamma|x'|^2$ in $\Omega$ for any $(x_N,x')\in\rr^N$ very
close to origin. Such points belong to the  paraboloid $P_\gamma$
defined by
\begin{equation}\label{eq115}
P_\gamma=\{x=(x',x_N)\in \rr^N\ |\  x_N>\gamma|x'|^2\}.
\end{equation}
 Due to the considerations above, we distinguish the following four main situations.

\begin{enumerate}[C1.]
\item\label{eqq2} {\bf The elliptic case:} There exists  $\gamma>0$ such that
 $\Omega\subset P_\gamma$. In other words, $\Omega$
 lies in the corresponding elliptic paraboloid $P_\gamma$
(see Figure \ref{fig4}, top left).
 \item\label{eqq1}  {\bf The cylindrical case: }  $\Omega\subset P_0$, where $P_0=\rr_{+}^{N}$.
(see Figure \ref{fig4}, top right).
\begin{obs}\label{obs1}
 In cases $\textsc{C}\ref{eqq2}$ and $\textsc{C}\ref{eqq1}$,  $\Omega$
lies in $\rr_{+}^{N}$
(see Figure \ref{fig4}, top). From this point of view they may be
analyzed together and the results that are true in
$\textsc{C}\ref{eqq1}$ are also valid in $\textsc{C}\ref{eqq2}$.
However, we analyze them separately because we present two
independent tools to treat each of them.
\end{obs}

 \item\label{eqq3} {\bf The locally elliptic case:} In this case,
$\Omega$ does not lie in $\rr_{+}^{N}$, but  this happens near  the
origin. More precisely, we suppose the existence of
$\gamma_{\textrm{local}}\geq 0$ such that $\Omega\subset
P_{\gamma_{\textrm{local}}}$ near  the origin. Away from the origin
we suppose that  there exists  $\gamma<0$ such that $\Omega\subset
P_{\gamma}$ (see Figure \ref{fig4}, bottom left).

\item\label{eqq4} {\bf The hyperbolic case:} This corresponds to the situation
when $\Omega$ has a hyperbolic geometry near the origin $x=0$.
Therefore, we suppose that $\Omega\subset P_\gamma$ for some
negative $\gamma<0$. (see Figure \ref{fig4}, bottom right).
\end{enumerate}

\begin{figure}[h]
\begin{center}
\setlength{\unitlength}{.4cm}
\begin{picture}(16,10)
\linethickness{0.3mm}
  \put(2,3){\vector(1,0){11}}
  \put(7,1.3){\vector(0,1){7.2}}
  \textcolor{blue}{\qbezier(3.2,6.8)(7.25,-1)(11,7)}
   \textcolor{red}{\qbezier(5,4)(7,1.9)(9,4.3)
   \qbezier(5,4)(2.9,6.5)(6.5,7.1)
   \qbezier(6.5,7.1)(11.9,8.6)(9,4.3)}
   \put(7,9.4){\makebox(0,0){\emph{\bf C1: The elliptic case $\gamma\geq0$ }}}
   \put(7.4,4.5){\makebox(0,0){\emph{$\mathbf{\Omega}$}}}
   \put(11.7,2.5){\makebox(0,0){$\mathbf{x_N=0}$}}
   \put(5.4,8.2){\makebox(0,0){$\mathbf{x'=\bf{0}}$}}
   \put(3.1,5.6){\makebox(0,0){$\mathbf{P_\gamma}$}}
   \put(12.7,6.2){\makebox(0,0){$\mathbf{x_N=\gamma|x'|^2}$}}
    \put(6.2,2.3){\bf{0}}
    \end{picture}\quad
 \begin{picture}(14,10)
\linethickness{0.3mm}
 \textcolor[rgb]{0.00,0.00,0.50}{ \put(2.5,3){\vector(1,0){10}}}
  \put(7,1){\vector(0,1){7.8}}
   \textcolor[rgb]{1.00,0.00,0.00}{\qbezier(5,4)(5.8,3.05)(6.5,3)
   \qbezier(5,4)(3,7)(6.5,7.4)
   \qbezier(6.5,7.4)(11.5,9)(9,3)
 }  \put(7,9.5){\makebox(0,0){\emph{\bf C2: The cylindrical case $\gamma=0$ }}}
   \put(7.6,4.5){\makebox(0,0){\emph{$\mathbf{\Omega}$}}}
   \put(11.5,2.5){\makebox(0,0){$\mathbf{x_N=0}$}}
   \put(5.7,8.3){\makebox(0,0){$\mathbf{x'=\bf{0}}$}}
   \put(6.5,2.3){\bf{0}}
  \end{picture}\\
\begin{picture}(16,10)
\linethickness{0.3mm}
  \put(2,3){\vector(1,0){11.3}}
  \put(7.15,1){\vector(0,1){7.8}}
  \textcolor[rgb]{0.25,0.00,0.25}{\qbezier(3,7)(7.3,-1.1)(11,7)}
  \textcolor{red}{
  \qbezier(5,4.1)(7.2,1.8)(9,4.3)
   \qbezier(5,4.1)(4,5)(3,1.5)
   \qbezier(3,1.5)(2.2,0)(4,5)
   \qbezier(11,1.5)(11.9,0)(10,5)
    \qbezier(4,5)(7,11)(10,5)
   \qbezier(9,4.3)(10,5.3)(11,1.5)}
  \textcolor{blue}{ \qbezier(1.9,0.5)(7,5.55)(11.8,0.3)}
   \put(0.5,0){$\mathbf{P_{\gamma}}$}
   \put(10,-0.3){$\mathbf{x_N=\gamma|x'|^2}$}
   \put(7,9.5){\makebox(0,0){\emph{\bf C3: The locally elliptic case $\gamma_{\textrm{loc}}\geq0$}}}
   \put(7.4,4.5){\makebox(0,0){\emph{$\mathbf{\Omega}$}}}
   \put(6.8,2.3){\bf{0}}
   \put(12.3,2.7){\makebox(0,0){$\mathbf{x_N=0}$}}
   \put(5.5,8.4){\makebox(0,0){$\mathbf{x'=\bf{0}}$}}
   \put(2.3,6){\makebox(0,0){$\mathbf{P_{\gamma_{\textrm{\bf loc}}}}$}}
   \put(13,6.5){\makebox(0,0){$\mathbf{x_N=\gamma_{\textrm{\bf loc}}|x'|^2}$}}
  \end{picture}\quad
\begin{picture}(14,10)
\linethickness{0.3mm}
  \put(2,3){\vector(1,0){11}}
  \put(7,1){\vector(0,1){7}}
  \textcolor{blue}{\qbezier(2.6,1.2)(6.9,4.64)(11.5,1.4)}
   \textcolor{red}{\qbezier(4,2.4)(7,3.65)(10,2.3)
   \qbezier(4,2.4)(3.15,1.9)(5.1,5)
   \qbezier(10,2.3)(11,2)(9,5)
   \qbezier(5.1,5)(7.1,7.5)(9,5)}
   \put(2,1.9){$\mathbf{P_\gamma}$}
   \put(6.9,2.3){\bf{0}}
   \put(10,0.7){$\mathbf{x_N=\gamma|x'|^2}$}
   \put(7,9){\makebox(0,0){\emph{\bf C4: The hyperbolic case $\gamma<0$}}}
   \put(7.4,4.5){\makebox(0,0){\emph{$\mathbf{\Omega}$}}}
   \put(12.1,2.6){\makebox(0,0){$\mathbf{x_N=0}$}}
   \put(5.6,7.7){\makebox(0,0){$\mathbf{x'=\bf{0}}$}}
  \end{picture}
\caption{\label{fig4}}
\end{center}
\end{figure}

In the sequel we need  the following technical lemma whose proof is given at the end of the section. 
\begin{lema}\label{l11}
Let $\Omega\subset P_\gamma$ be a domain fulfilling one of the
conditions $\textsc{C}1-\textsc{C}4$ in Figure \ref{fig4}, for some
constant $\gamma\in \rr$.
 Given
$N\geq 2$ and  $v\in H_{0}^{1}(\Omega)$,  for any constant $C\in
\rr$, the function $u$
 defined by
\begin{equation}\label{equ5}
u(x)=\frac{v(x)}{(x_N-\gamma|x'|^2)}|x|^{C}.
\end{equation}
fulfills the following  identity:
\begin{eqnarray}\label{equ6} \int_{\Omega}|\nabla
v|^2dx&=&\int_{\Omega}(x_N-\gamma|x'|^2)^2|x|^{-2C}|\nabla
u|^2dx+(CN-C^2)\int_{\Omega}\frac{v^2}{|x|^2}dx\nonumber\\
&+&2\gamma\int_{\Omega}\Big((N-1)|x|^2-C|x'|^2\Big)(x_N-\gamma|x'|^2)|x|^{-2C-2}u^2dx.
\end{eqnarray}
\end{lema}

\subsection{Proof of main results}
Next, we consider  $\Omega$ to be a domain which satisfies the case
C\ref{eqq2} in Figure \ref{fig4} (top, left). The main result in
this case consists in an improvement of the Hardy-Poincar\'{e}
inequality \eqref{HPI} shown in \cite{musina}. The proof is based on
Lemma \ref{l11}. More precisely, we have

\begin{Th}\label{tu2} Let $N\geq 3$. Assume that $\Omega$
satisfies the condition $\textsc{C}1$.  Then, for all $v\in
H_0^1(\Omega)$ there exists a positive constant $C(\Omega,\gamma)$
such that
\begin{equation}\label{equ17}
\int_{\Omega}|\nabla v|^2dx-
\frac{N^2}{4}\int_{\Omega}\frac{v^2}{|x|^2}dx\geq
C(\Omega,\gamma)\int_{\Omega}\frac{v^2}{|x|}dx.
\end{equation}
When $N=2$, the following weaker inequality holds
\begin{equation}\label{eqq5}
\int_{\Omega}|\nabla v|^2dx\geq
\frac{N^2}{4}\int_{\Omega}\frac{v^2}{|x|^2}dx.
\end{equation}
\end{Th}
\begin{proof}[Proof of Theorem \ref{tu2}]
We put $C=N/2$ in the identity (\ref{equ6}) of Lemma \ref{l11}.
Taking into account that
$$\max_{C\in \rr}\big\{CN-C^2\big\}=\big\{CN-C^2\big\}\Big|_{C=N/2}=\frac{N^2}{4},$$
we obtain
\begin{equation}\label{eqqu1}
\into |\nabla v|^2dx\geq \frac{N^2}{4}\into
\frac{v^2}{|x|^2}dx+2\gamma\into
\Big((N-1)|x|^2-\frac{N}{2}|x'|^2\Big)(x_N-\gamma|x'|^2)|x|^{-N-2}u^2dx
\end{equation}
Using that
$$(N-1)|x|^2-\frac{N}{2}|x'|^2\geq \frac{N-2}{2}|x|^2,$$
we get
\begin{equation}\label{equ18}
\int_{\Omega}|\nabla v|^2dx\geq
\frac{N^2}{4}\int_{\Omega}\frac{v^2}{|x|^2}dx+\gamma(N-2)\int_{\Omega}\frac{v^2(x)}{x_N-\gamma|x'|^2}dx.
\end{equation}
We split the last term in two parts as follows:
\begin{align*}
\int_{\Omega}\frac{v^2}{x_N-\gamma|x'|^2}dx&=& \int_{\{x\in \Omega,\
|x'|\leq 1/\gamma\}}\frac{v^2}{x_N-\gamma|x'|^2}dx+\int_{\{x\in
\Omega,\ |x'|\geq 1/\gamma\}}\frac{v^2}{x_N-\gamma|x'|^2}dx:=I_1+I_2
\end{align*}
In the first term, using that $|x'|\leq 1/\gamma$ implies
$x_N-\gamma|x'|^2\leq 2|x|$, we obtain
\begin{align*}
I_1\geq \frac{1}{2}\int_{\{|x'|\leq1/\gamma\}}\frac{v^2}{|x|}dx.
\end{align*}
Using the notation $R_\Omega=\sup_{x\in \overline{\Omega}}|x|$ we
have $x_N-\gamma|x'|^2\leq R_\Omega+\gamma R_{\Omega}^2.$ Thus, for
the second term we obtain
\begin{align*}
 I_2 \geq  \frac{1}{R_\Omega+\gamma
R_{\Omega}^{2}}\int_{\{|x'|\geq 1/\gamma\}} v^2dx\geq
\frac{1}{\gamma(R_\Omega+\gamma R_{\Omega}^{2})}\int_{\{|x'|\geq
1/\gamma\}}\frac{v^2}{|x|}dx
\end{align*}
Combining these two lower bounds we get
\begin{align*}
I_1+I_2\geq \min\Big\{\frac{1}{2}, \frac{1}{\gamma(R_\Omega+\gamma
R_{\Omega}^{2})}\Big\}\int_{\Omega}\frac{v^2}{|x|}dx,
\end{align*}

and this, together with (\ref{equ18}) yields (\ref{equ17}). For
$N=2$, (\ref{eqq5}) holds easily from (\ref{equ18}).
\end{proof}

Lemma \ref{l11}  does not provide sufficient information for
$\gamma=0$. However, using spherical harmonics decomposition, we can
extend and improve the result of Theorem \ref{tu2} to the case
$\gamma\geq 0$ as follows.

\begin{Th}\label{tu3} Let $N\geq 2$, and $\Omega\subset \rr^N$ be such that
the condition $\textsc{C}2$ is satisfied in Figure \ref{fig4} (top,
right).  If $L$ is a positive number such that $L> \sup_{x\in
\overline{\Omega}}|x|$, then for any $v\in H_0^1(\Omega)$,
\begin{equation}\label{equ21}
\int_{\Omega}|\nabla v|^2dx\geq
\frac{N^2}{4}\int_{\Omega}\frac{v^2}{|x|^2}dx+\frac{1}{4}\int_{\Omega}\frac{v^2}{|x|^2
\log^2(L/|x|)}dx,
\end{equation}
\end{Th}

 The following lemma will be necessary  in the proof of Theorem
 \ref{tu3}.
\begin{lema}\label{lema2}
\label{l13} Let $L>R>0$. Then
\begin{equation}\label{eq121}
\intR (w'(r))^2r dr \geq \frac{1}{4}\intR
\frac{w^2}{r^2\log^2(L/r)}rdr,\quad \forall \quad w \in
H_{0}^{1}(0,R).
\end{equation}
\end{lema}

The proof of Lemma \ref{lema2} is given at the end of the section.

\begin{proof}[Proof of Theorem \ref{tu3}]
Firstly, let us  set
$R=R_{\Omega}:=\sup_{x\in\overline{\Omega}}|x|$. such that
$\Omega\subset B_R^+$ where $B_R^+$ is the half ball of radius $R$
$$B_R^{+}:=\{x\in \rr^N,\ |x|\leq R,\ x_N\geq 0\}.$$
We consider also the lower half ball of radius $R$,
$$B_R^{-}:=\{x\in \rr^N,\
|x|\leq R,\ x_N\leq 0 \}.$$

The proof follows the idea of decomposition in spherical harmonics
(see \cite{MR1760280}). By a density argument we can consider $v\in
C_{0}^{1}(B_R^+)$. Building the odd extension
\begin{equation}\label{equ22}u(x)=u(x_1,x_2,...,x_N):=\left\{\begin{array}{ll}
  v(x_1,x_2,...,x_N), & x\in B_R^+,  \\
  -v(x_1,x_2,...,-x_N), & x\in B_R^-, \\
\end{array}\right.
\end{equation}
we obtain $u\in C_{0}^{1}(B_R)$ and moreover,
\begin{align}\label{equ23}
\int_{B_R^+}|\nabla v|^2dx=\frac{1}{2}\int_{B_R}|\nabla u|^2dx,
\end{align}
\begin{align}
\int_{B_R^+}\frac{v^2}{|x|^2}dx=\frac{1}{2}\int_{B_R}\frac{u^2}{|x|^2}dx.
\end{align}
Next we note that
\begin{equation*}\label{equ24}
\int_{S^{N-1}}u(r,\sigma)d\sigma=0, \quad \forall \quad r\in [0,1].
\end{equation*}
Consider the expansion of $u$ in spherical harmonics
\begin{equation}\label{eq122}
u(x)=u(r,\sigma)=\sum_{k=0}^{\infty}u_k(r)f_k(\sigma).
\end{equation}
Here $(f_k)_{k\geq 0}$ is an ortonormal basis of $L^2(S^{N-1})$
constituted by the eigenvectors of the spherical Laplacian
$\Delta_{S^{N-1}}$ with the corresponding eigenvalues
$c_k=k(N+k-2)$, $k\geq 0$. Here $S^{N-1}$ is the unit sphere with
$(N-1)$-dimensional Hausdorff measure $N\omega_N$, where $\omega_N$
is  the Lebesgue
measure of the unit ball. It is well-known that $f_0$ is a constant. 
Integrating (\ref{eq122}) on $S^{N-1}$ we get
$$u_0(r)=\int_{S^{N-1}}u(r,\sigma)f_0(\sigma)d\sigma=f_0(\sigma)\int_{S^{N-1}}u(r,\sigma)d\sigma=0,$$
Therefore
\begin{equation}\label{eq123}
u(x)=u(r,\sigma)=\sum_{k=1}^{\infty}u_k(r)f_k(\sigma).
\end{equation}
 and  by  Plancherel identity we have
\begin{equation}\label{equ25}
\int_{B_R}
u^2dx=N\omega_N\sum_{k=1}^{\infty}\int_{0}^{R}|u_k(r)|^2r^{N-1}dr.
\end{equation}
 Using the representation of the Laplace operator in spherical coordinates
$$\Delta= -\pa_{r}^{2}-\frac{N-1}{r}\pa_{r}u-\frac{1}{r^2}\Delta_{S^{N-1}},$$
 we get
\begin{align}\label{equ26}
\int_{B_R}|\nabla u|^2dx=N\omega_N \sum_{k=1}^{\infty}\int_{0}^{R}
\Big[|u_{k}^{'}|^2+c_k\frac{u_{k}^{2}(r)}{r^2}\Big]r^{N-1}dr.
\end{align}
Let us denote $w_k(r)=u_k(r)r^{\frac{N-2}{2}}$ and $C_N:=(N-2)^2/4$.
Then by (\ref{equ26}) we have
\begin{align}\label{eq111}
\int_{B_R}|\nabla
u|^2dx&-\frac{N^2}{4}\int_{B_R}\frac{u^2}{|x|^2}dx=
N\omega_N\sum_{k=1}^{\infty}\int_{0}^{R}\Big[|u_{k}^{'}|^2-C_N\frac{u_k^2}{r^2}\Big]r^{N-1}dr+\nonumber\\
&+N\omega_N\sum_{k=1}^{\infty}\Big(c_k-(N-1)\Big)\int_{0}^{R}\frac{u_{k}^{2}}{r^2}r^{N-1}dr\nonumber\\
&\geq
N\omega_N\sum_{k=1}^{\infty}\int_{0}^{R}\Big[|u_{k}^{'}|^2-C_N\frac{u_k^2}{r^2}\Big]r^{N-1}dr\nonumber\\
&=N\omega_N\sum_{k=1}^{\infty}\int_{0}^{R}|w_k'(r)|^2rdr
\end{align}
Hence, by (\ref{eq111}) and Lemma \ref{l13} we have
\begin{align}\label{eqn4}
\int_{B_R}|\nabla u|^2dx-\frac{N^2}{4}\int_{B_R}\frac{u^2}{|x|^2}dx
&\geq \frac{N\omega_N}{4}\sum_{k=1}^{\infty}
\int_{0}^{R}\frac{w_{k}^{2}(r)}{r^2
\log^2(L/r)}rdr\nonumber\\
&=\frac{N\omega_N}{4}
\sum_{k=1}^{\infty}\int_{0}^{R}\frac{u_{k}^{2}(r)}{r^2 \log^2(L/r)}
r^{N-1}dr.
\end{align}
On the other hand
\begin{align}\label{eq106}
N\omega_N\sum_{k=1}^{\infty}\int_{0}^{R}\frac{u_k^2(r)}{(r^2\log^2(L/r)}r^{N-1}dr
=\int_{B_R}&\frac{u^2}{|x|^2\log^2(L/|x|)}dx.
\end{align}
 By (\ref{eqn4}), (\ref{eq106}) and undoing the variables, the proof is completed.
\end{proof}
\begin{Th}\label{tu7} Let  $N\geq 2$ and $\Omega$ be a domain
satisfying the case $\textsc{C}\ref{eqq3}$ as in Figure \ref{fig4}
(bottom, left). Then, for any $v\in \hoi$, there exists a constant
$C(\Omega)$ such that
\begin{equation}\label{equ36}
C(\Omega)\int_{\Omega}v^2dx+\int_{\Omega}|\nabla
v|^2dx\geq\frac{N^2}{4}\int_{\Omega}\frac{v^2}{|x|^2}dx+\frac{1}{4}\int_{\Omega}\frac{v^2}{|x|^2\log^2(L/|x|)}dx,
\end{equation}
where $L> \sup_{x\in \overline{\Omega}}|x|$. Moreover, \eqref{equ36}
is optimal in the sense that the term corresponding to the
$L^2$-norm on the left hand side term  cannot be disregarded.
\end{Th}
{\bf Sketch of the proof of Theorem \ref{tu7}:}  We apply  a
standard cut-off argument so that the function $v$ can be split as
$v=v_1+v_2$ where $v_1$ lies near the singularity and
 $v_2$ is supported away from it. In the neighborhood of $x=0$,
  we can apply  the improved inequality of Theorem \ref{tu3} corresponding to $v_1$.
Outside the origin there are no singularities so that the potential
$1/|x|^2$ that appears in the inequality, is bounded by a constant
depending only on  $\Omega$ and the profile of the cut-off function.
This fact makes the quantity $\int v_2^2/|x|^2dx$ to be bounded from
above, up to a constant,  by $\int v_2^2dx$. There is also an
intermediate zone that we have to deal with, and more precisely
where the profile of the cut-off functions has the gradient
different by zero. In that part, it suffices to show that the cross
term $\int \nabla v_1\cdot \nabla v_2$ is bounded from below, up to
a constant, by
 $\int v^2dx$. Gluing these,  the proof of \eqref{equ36} ends. We skip all
the computations of the proof but for more details of a cut-off
technique see e.g. \cite{MR1760280}, pp. 111.

 Then, the necessity of the
term in $L^2$-norm on the left hand side of \eqref{equ36} is a
consequence of Proposition \ref{hardysub}.

 In the sequel we consider   a domain $\Omega\subset
\rr^N$ as in  Figure \ref{fig4} (bottom, right). The result we
obtain is stated as follows.
\begin{Th}\label{ta5} Let $N\geq 2$ and assume that $\Omega$
satifies the condition $\textsc{C}4$ as in Figure \ref{fig4}
(bottom, right).  For any $\eps>0$, $\eps <<1$, there exists a
constant $C(\Omega,\eps)$ such that the following inequality holds:
    \begin{equation}\label{equ46}
 C(\Omega,\eps)\into v^2dx+\into |\nabla v|^2dx \geq
 \Big(\frac{N^2}{4}-\eps\Big)\into\frac{v^2}{|x|^2} dx, \quad \forall \quad v\in
 \hoi.
    \end{equation}
Moreover, \eqref{equ46} is optimal in the sense that  the term in
$L^2$-norm of the left hand side term cannot be disregarded.
\end{Th}
\begin{obs}
Theorem \ref{ta5} is stated also in \cite{musina}. The authors in
\cite{musina} omit to show the continuous dependence of the Hardy
constant in cones which plays a crucial role in the proof of this
theorem. To make the things clear, we give a rigorous proof of this
last result in Section \ref{Sec2}.
\end{obs}
{\bf Sketch of the proof of Theorem \ref{ta5}}.  The proof is based
on
  local approximations of $\Omega$ around the origin, by conical
  sectors. We consider sectors that approximate the hyperplane
   $\gamma=0$ from below. In this analysis we use Hardy inequalities in
   cones  and the continuous dependence of the Hardy constant, more precisely,  Corollary \ref{obaa2} in Section 3 below. This allows improving Hardy inequalities
    near the origin. Cut-off arguments allow to
    glue Hardy inequalities derived near the origin with terms in $L^2$-norm,
    provided
   that the potential $1/|x|^2$ is bounded far from the
   singularity $x=0$. Basically,   this proves \eqref{equ46}. For further details of a cut-off
   technique we refer to \cite{MR1760280}, pp. 111.

  The necessity of the term in
$L^2$-norm on the left hand side of \eqref{equ46} is a consequence
of Proposition \ref{hardysub} stated below.

\begin{prop}\label{hardysub}
There exist  smooth bounded open sets $\Omega\subset \rr_{+}^{N}$,
$N\geq 2$, satisfying either  $\textsc{C}\ref{eqq3}$ or
$\textsc{C}\ref{eqq4}$ and such that
\begin{equation}\label{cont}
\mu(\Omega)< \frac{N^2}{4}.
\end{equation}
\end{prop}
\begin{proof}[Proof of Proposition \ref{hardysub}]
From the characterization of the first eigenvalue one can show the
strict anti-monotonicity
\begin{equation}\label{eigenvalues7}
\mathcal{D}_1\subset\subset \mathcal {D}_2\Rightarrow
\lambda(\mathcal{D}_1)> \lambda(\mathcal{D}_2)
\end{equation}
Next we take a cone $\mathcal{C}$ strictly larger than
$\rr_{+}^{N}$. From \eqref{eigenvalues7} and \eqref{constantcone} we
obtain
$$\mu(\mathcal{C})<\mu(\rr_{+}^{N})=\frac{N^2}{4}.$$
Therefore, there exists $u\in C_{0}^{\infty}(\mathcal{C})$ such that
$$\frac{\int_{\mathcal{C}}|\nabla u|^2dx}{\int_{\mathcal{C}}u^2/|x|^2dx } < \frac{N^2}{4}.$$
Denote $K:=\overline{\textrm{supp}u}$. Then $K\subset\subset
\mathcal{C}$ and dist$(K, \pa \mathcal{C})>0$. Hence, we can build
an open set $\Omega$ satisfying either  C\ref{eqq3} or C\ref{eqq4},
such that $K\subset \Omega\subset \mathcal{C}.$ Hence, $u\in
C_{0}^{\infty}(\Omega)$ and we get that
$$\mu(\Omega)\leq \frac{\int_{\mathcal{C}}|\nabla u|^2dx
}{\int_{\mathcal{C}}|u|^2/|x|^2dx}<\frac{N^2}{4}.$$ The proof is
completed.
\end{proof}

\subsection{Proofs of useful Lemmas}

\begin{lema}\label{lu1}
Let $N\geq 2$ and $\Omega\subset \rr^N$ satisfying one of the
conditions $\textsc{C}1-\textsc{C}4$.  If  $v\in H_{0}^{1}(\Omega)$
then
\begin{equation}\label{eqn8}
\lim_{x\rightarrow \pa P_\gamma}\frac{v^2(x)}{x_N-\gamma|x'|^2}=0.
\end{equation}
\end{lema}
\begin{proof}[Proof of Lemma \ref{lu1}]
Extending $v$ with 0 outside $\Omega$ we get $v(x',\gamma|x'|^2)=0$.
Then
$$0<\frac{v^2(x)}{x_N-\gamma|x'|^2}\leq
\int_{\gamma|x'|^2}^{x_N}\Big|\frac{\pa v}{\pa
y_N}(x',y_N)\Big|^2dyn
$$
which converges to 0 when $x_N\rightarrow \gamma|x'|^2$. This is due
to the properties of Lebesque integral and the fact that  $\frac{\d
v }{\d y_N}(x',\cdot)$ belongs to $ L^2(\gamma|x'|^2, x_N)$.
\end{proof}

\begin{proof}[Proof of Lemma \ref{l11}]
Firstly
\begin{equation}\label{eq33}\nabla v=(x_N-\gamma|x'|^2)|x|^{-C}\nabla u+u\nabla
[(x_N-\gamma|x'|^2)|x|^{-C}],
\end{equation}
\begin{align*}
\nabla [(x_N-\gamma|x'|^2)|x|^{-C}]&=\sum_{i=1}^{N-1}[-2\gamma
x_i|x|^{-C}-C(x_N-\gamma|x'|^2)|x|^{-C-2}x_i]e_i+\\
&+[|x|^{-C}-Cx_N(x_N-\gamma|x'|^2)|x|^{-C-2}]e_N,
\end{align*}
where $\{e_i\}_{i=1,N}$ is canonical basis of $\rr^N$. Then
\begin{align}\label{equ7}
|\nabla[(x_N-\gamma|x'|^2)|x|^{-C}]|^2&=(1+4\gamma^2|x'|^2)|x|^{-2C}+(C^2-2C)(x_N-\gamma|x'|^2)^2|x|^{-2C-2}+\nonumber\\
&+2\gamma C|x'|^2(x_N-\gamma|x'|^2)|x|^{-2C-2}
\end{align}
and
\begin{align}\label{equ8}
\textrm{div}\Big\{(x_N-\gamma|x'|^2)|x|^{-C}&\nabla
[(x_N-\gamma|x'|^2)|x|^{-C}]\Big\}\nonumber\\
&=(2C^2-CN-2C)(x_N-\gamma|x'|^2)^2|x|^{-2C-2}+(1+4\gamma^2|x'|^2)|x|^{-2C}
\nonumber\\
&-2[\gamma(N-1)|x|^2-2\gamma C|x'|^2](x_N-\gamma|x'|^2)|x|^{-2C-2}.
\end{align}
Using the formulas from above and integrating by parts we obtain
\begin{align}\label{equ9}
\int_{\Omega}|\nabla v|^2dx&=\int_{\Omega}|\nabla u|^2(x_N-\gamma
|x'|^2)^2|x|^{-2C}dx+\int_{\Omega}|\nabla
[(x_N-\gamma|x'|^2)|x|^{-C}]|^2u^2dx+\nonumber\\
&+2\int_{\Omega}(x_N-\gamma|x'|^2)|x|^{-C}\nabla[(x_N-\gamma|x'|^2)|x|^{-C}]u\nabla
udx\nonumber\\
&=\int_{\Omega}|\nabla u|^2(x_N-\gamma
|x'|^2)^2|x|^{-2C}dx+\int_{\Omega}|\nabla
[(x_N-\gamma|x'|^2)|x|^{-C}]|^2u^2dx\nonumber\\
&+\int_{\pa
\Omega}u^2(x_N-\gamma|x'|^2)|x|^{-C}\nabla[(x_N-\gamma|x'|^2)|x|^{-C}]\cdot
\nu
d\sigma-\nonumber\\
&-\int_{\Omega}\textrm{div}\Big\{(x_N-\gamma|x'|^2)|x|^{-C}
\nabla[(x_N-\gamma|x'|^2)|x|^{-C}]\Big\}u^2dx.
\end{align}
Estimating the expression of the gradient in (\ref{equ7}) we get
\begin{align}\label{equ10}
|\nabla[(x_N-\gamma|x'|^2)|x|^{-C}]|^2
\leq C(\gamma,\Omega)|x|^{-2C},
\end{align}
where 
$C(\gamma, \Omega)$ is a suitable positive  constant. Therefore,
from (\ref{equ10}) we have
\begin{equation}\label{equ11}\Big|u^2(x_N-\gamma|x'|^2)|x|^{-C}\nabla[(x_N-\gamma|x'|^2)|x|^{-C}]\Big|\leq
C\frac{v^2(x)}{x_N-\gamma|x'|^2},
\end{equation}
for some constant $C$.  According to  Remark \ref{lu1} and
(\ref{equ11}) the boundary term of (\ref{equ9}) vanishes and we
obtain the new identity
\begin{align}\label{equ12}
\int_{\Omega}|\nabla v|^2dx&=\int_{\Omega}|\nabla u|^2(x_N-\gamma
|x'|^2)^2|x|^{-2C}+\int_{\Omega}|\nabla
[(x_N-\gamma|x'|^2)|x|^{-C}]|^2u^2dx-\nonumber\\
&-\int_{\Omega}\textrm{div}\Big\{(x_N-\gamma|x'|^2)|x|^{-C}
\nabla[(x_N-\gamma|x'|^2)|x|^{-C}]\Big\}u^2dx.
\end{align}
By (\ref{equ7}), (\ref{equ8}) and (\ref{equ12}) we have the identity
(\ref{equ6}). With this the proof of Lemma \ref{l11} ends.
\end{proof}

\begin{proof}[Proof of Lemma \ref{l13}]
With the change of variables $w(r)=v(r)\log^{1/2}(L/r)$ we have
$$|w'(r)|^2=\frac{1}{4r^2}\log^{-1}(L/r)v^2+v_{r}^{2}\log(L/r)-\frac{1}{r}vv_r.$$
Therefore, due to the zero boundary conditions, we obtain
\begin{eqnarray*}\label{eqn1}
\intR |w'(r)|^2rdr&=& \frac{1}{4}\intR
\frac{w^2}{r^2\log^2(L/r)}rdr+\intR v_r^2\log(L/r)rdr -\intR v
v_rdr\nonumber\\
&=&\frac{1}{4}\intR \frac{w^2}{r^2\log^2(L/r)}rdr+\intR
v_r^2\log(L/r)rdr\nonumber\\
&\geq &\frac{1}{4}\intR \frac{w^2}{r^2\log^2(L/r)}rdr.
\end{eqnarray*}
and Lemma \ref{l13} holds true.
\end{proof}

\section{Inequalities in cones}\label{Sec2}
 Firstly, let us consider a Lipschitz connected cone
$\mathcal{C} \subset \rr^{N}\setminus \{0\}$ with the vertex at
zero. Let $D\subset S^{N-1}$ be the Lipschitz domain such that
$$\mathcal{C}=\{(r, \omega) \  |\   r\in (0, \infty ),\  \omega \in \mathcal{D}\}$$
 Let $\mu(\mathcal{C})$ be the best constant in the Hardy inequality.
 Then (cf. \cite{tintarevpinchover})
\begin{equation}\label{constantcone}
\mu(\mathcal{C})=\frac{(N-2)^2}{4}+\lambda_1(\mathcal{D})
\end{equation}
where $\lambda_1(\mathcal{D})$ is the Dirichlet principal eigenvalue
of the spherical Laplacian $-\Delta_{S^{N-1}}$ on $\mathcal{D}$. In
2-d it is well-known that (e.g. \cite{caldiroli2})
$$\lambda_1(\gamma):=\lambda_1(0,\gamma)=\pi^2/\gamma^2,$$ where
$\gamma$ is the slot of the conical sector $\mathcal{C}_\gamma=\{(r,
\omega) \  |\   r\in (0, \infty ),\ \omega \in (0,\gamma)\}$ (see
Figure \ref{fig8} below).
 \begin{figure}[h]
\begin{center}
\setlength{\unitlength}{.5cm}
\begin{picture}(16,10)
\linethickness{0.3mm}
  \put(2,3){\vector(1,0){5}}
  \textcolor[rgb]{0.00,0.00,0.50}{\put(7,3){\vector(1,0){6}}}
  \put(7,1.3){\vector(0,1){7.2}}
   \put(11.5,4.2){\makebox(0,0){\emph{$\mathbf{\mathcal{C}_\gamma}$}}}
   \put(12.5,2.6){\makebox(0,0){$\mathbf{x_1}$}}
   \put(6.2,8){\makebox(0,0){$\mathbf{x_2}$}}
   \put(8.7,3.35){\bf\makebox(0,0){$\mathbf{\gamma}$}}
    \put(6.35,2.4){\bf{0}}
    \textcolor[rgb]{1.00,0.00,0.00}{\qbezier(7.8,3.35)(8, 3.3)
          (7.9, 3)}
   \textcolor[rgb]{0.00,0.00,0.25}{\put(6.8,3){\bf\vector(2,1){5.5}}}
    \end{picture}
\caption{\label{fig8} The 2-d conical sector with the aperture
$\gamma$.}
\end{center}
\end{figure}
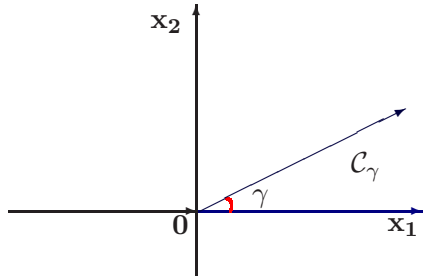
  In higher dimensions $N\geq 3$, by our knowledge,
$\lambda_1(\mathcal{D})$ is well-known only in the case where
$\mathcal{D}$ is the semi-sphere $S_{+}^{N-1}$ mapped in the upper
half space $\rr_{+}^{N}$. More precise,
$\lambda_1(S_{+}^{N-1})=N-1$. The half space $\rr_{+}^N$ corresponds
to the conical sector of slot $\gamma=\pi/2$ (see Figure \ref{figg8}
below).

The aim of this section is mainly devoted to find lower bounds for
$\lambda_1(\mathcal{D})$ in higher dimensions $N\geq 3$. In that
sense,  the definition of a cone in polar coordinates will be used.
\subsection{The $N-d$ case, $N\geq3$}\label{subs1}
For $0<\gamma<\pi$  we define the $N$-dimensional cone, with  slot
$\gamma$, denoted by $\mathcal{C}_\gamma$ (Figure \ref{fig8}),
consisting in all $x=(x_1,x_2,...,x_N)\in\rr^N$ such that, in
spherical coordinates  (cf. \cite{SteinS}, pp. 293),
\begin{equation}\label{eqa14}
\mathcal{C}_\gamma:\left\{\begin{array}{ll}
  x_1=r\sin \theta_1\sin \theta_2\ldots\sin \theta_{N-2}\cos_{N-1} \\
  x_2=r\sin \theta_1\sin \theta_2\ldots\sin \theta_{N-2}\sin_{N-1} \\
  \vdots\\
  x_{N-1}=r\sin \theta_1\sin \theta_2 \\
  x_N=r\cos \theta_1 \\
\end{array}\right.
\end{equation}
with  $r>0$ and
\begin{equation}\label{equ49}
\left\{\begin{array}{ll}
  0<\theta_1\leq \gamma, & \\
  0\leq \theta_i\leq \pi, & \textrm{ for } 2\leq i\leq N-2,  \\
  0\leq \theta_{N-1}\leq 2\pi. & \\
\end{array}\right.
\end{equation}

\begin{figure}[h]
\begin{center}
\setlength{\unitlength}{.5cm}
\begin{picture}(16,10)
\linethickness{0.3mm}
  \put(2,3){\vector(1,0){5}}
  \put(7,3){\vector(1,0){6}}
  \put(7,1.3){\vector(0,1){8}}
   \put(7,3){\vector(-1,-1){4}}
   \put(6.5,5.5){\makebox(0,0){\emph{$\mathcal{C}_\gamma$}}}
   \put(12.6,2.5){\makebox(0,0){$\mathbf{x_1}$}}
   \put(3,0){\makebox(0,0){$\mathbf{x_2}$}}
   \put(6.2,9){\makebox(0,0){$\mathbf{x_3}$}}
   \put(7.5,4.8){\bf\makebox(0,0){$\mathbf{\gamma}$}}
   \put(8.9,7.6){\bf\makebox(0,0){$\mathbf{x}$}}
   \put(7.5,2){\bf\makebox(0,0){$\mathbf{\theta_2}$}}
\put(8.2,5.9){\bf\makebox(0,0){$\mathbf{r}$}}
    \put(6.35,2.4){\bf{0}}
    \put(9.3,7.25){\textcolor[rgb]{1.00,0.00,0.00}{\line(0,-1){5.1}}}
    \put(7.1,3){\textcolor[rgb]{1.00,0.00,0.00}{\line(2,4){2.4}}}
    \put(7,3){\textcolor[rgb]{1.00,0.00,0.00}{\line(3,-1){2.45}}}
           \textcolor[rgb]{0.00,0.00,0.50}{\qbezier(7,4)(7.2, 4.1)
          (7.45, 3.9)
          \qbezier(6.7,2.7)(7.1,2.5)
          (7.5,2.8)}
    \put(9.3,7.25){\textcolor[rgb]{1.00,0.00,0.00}{\line(-6,1){2.2}}}
   \put(7, 7.6){\circle*{0.3}}
   \put(9.3, 7.25){\textcolor[rgb]{1.00,0.00,0.00}{\circle*{0.3}}}
   \qbezier(10,7.6)(7, 6)
          (4, 7.6)
\qbezier(10,7.6)(7, 9)
          (4, 7.6)
   \put(7,3){\bf\vector(2,3){3.1}}
\put(7,3){\bf\vector(-2,3){3.1}}
    \end{picture}
\caption{\label{figg8} The cone with the slot $\gamma$}
\end{center}
\end{figure}
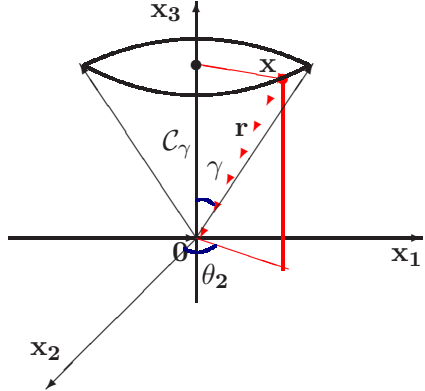

For simplicity we denote by $\lambda_1(\gamma):=\lambda_1(D_\gamma)$
the first Dirichlet
 eigenvalue of the spherical Laplacian on
$\mathcal{D}_\gamma:=\mathcal{C}_\gamma\cap S^{N-1}$. Then we have
\begin{equation}\label{ineq1}
\mu(\mathcal{C}_\gamma)=\frac{(N-2)^2}{4}+\lambda_1(\gamma)
\end{equation}

\subsection{Main results}
\begin{Th}\label{taa3}
Assume that $N\geq 3$.
\begin{enumerate}[(a).]
\item\label{dd} If $0<\gamma\leq \frac{\pi}{2}$ then
 $$\lambda_1(\gamma)\geq
\frac{(N-1)\pi^2}{4\gamma^2}.$$
\item\label{cc} For any $ \eps>0$, $\eps << 1$, there exists
$\delta=\delta(\eps)>0$ such that for all $ \quad \frac{\pi}{2}\leq
\gamma \leq \frac{\pi}{2}+\delta$
$$\lambda_1(\gamma)>N-1-\eps.$$
\end{enumerate}
\end{Th}
\begin{cor}[The continuous dependance of the Hardy constant]\label{obaa2}
\begin{equation}\label{eqaa16}
\lim_{\gamma\rightarrow
\pi/2}\mu(\mathcal{C_\gamma})=\mu(\mathcal{C}_{\pi/2})=\frac{N^2}{4}.
\end{equation}
\end{cor}

\begin{Th}\label{SL1}
Assume  $\gamma\in (0,\pi)$. Then it holds
\begin{equation}\label{firsteigenvalue2}
\Big(\frac{\sin
\gamma}{\gamma}\Big)^{N-2}\Big(\frac{B_1}{\gamma}\Big)^2\leq
\lambda_1(\gamma)\leq \Big(\frac{ \gamma}{\sin
\gamma}\Big)^{N-2}\Big(\frac{B_1}{\gamma}\Big)^2,
\end{equation}
where $B_1$ is the first positive zero of the Bessel function
$J_{\frac{N-3}{2}}$, of fractional order $(N-3)/2$.
\end{Th}

\begin{obs}[Asymptotic behavior]\label{asymptotic}
From Theorem \ref{SL1} above we get the  asymptotic formula

$$\lim_{\gamma\rightarrow 0}\frac{\lambda_{1}(\gamma)\gamma^2}{B_{1}^2}=1.$$
\end{obs}
\subsection{ Preliminaries} \vspace{0.5cm}

Let us  consider $u\in C_c^\infty(\con)$. Then, in polar coordinates
we have
\begin{equation}\label{ineq2}
u(r, \theta_1, \theta_2, \ldots, \theta_{N-3}, \theta_{N-2}) \in
C^\infty((0,\infty)\times (0, \gamma)\times
(0,\pi)\times\ldots\times (0, \pi)\times (0, 2\pi)),
\end{equation}
 vanishing in
the neighborhoods of $r=\infty$ and $\theta_1=\gamma$.

 The representation of the gradient in polar coordinates is
given by
\begin{align}\label{eqaa60}
|\nabla
u|^2&=|u_r|^2+\frac{u_{\theta_1}^2}{r^2}+\frac{u_{\theta_2}^2}{r^2\sin^2\theta_1}+\frac{u_{\theta_3}^2}{r^2\sin^2\theta_1\sin^2
\theta_2}+\nonumber\\ &+\ldots
+\frac{u_{\theta_{N-1}}^{2}}{r^2\sin^2\theta_1\sin^2\theta_2\ldots
\sin^2{\theta_{N-2}}}.
\end{align}
In other words,
$$|\nabla u|^2=|u_r|^2+\frac{u_{\theta_1}^2}{r^2}+\textrm{ positive terms}.$$
The determinant of  the Jacobian of the transformation has the form
$$J(r,\theta_1,\theta_2\ldots, \theta_{N-2})=r^{N-1}\sin^{N-2}
\theta_1\sin^{N-3} \theta_2\ldots \sin\theta_{N-2}.$$ To simplify
the notations, we define the integral in the variables
$\theta_2,\ldots, \theta_{N-2}$ as
$$\int:=\int_{0}^{\pi}\ldots \int_{0}^{\pi}\int_{0}^{2\pi}
\sin^{N-3}\theta_2\ldots \sin\theta_{N-2} \dtheta_{2} \ldots
\dtheta_{N-2}\dtheta_{N-1}.$$

For a fixed $\sigma\in S^{N-1}\cap \con$ the radial function
$r\mapsto u(\cdot,\sigma)$ satisfies the well-known Hardy inequality
$$
\int_{0}^{\infty} u_r^2 r^{N-1}dr\geq \Big(\frac{N-2}{2}\Big)^2
\int_{0}^{\infty} \frac{u^2}{r^2} r^{N-1}dr,$$ and therefore
\begin{align}\label{eqa55}
\int_{0}^{\infty}\int_{0}^{\gamma} \int u_r^2 r^{N-1}\sin^{N-2}
\theta_1 \dtheta_1dr \geq  \Big(\frac{N-2}{2}\Big)^2
\int_{0}^{\infty}\int_{0}^{\gamma} \int \frac{u^2}{r^2}
r^{N-1}\sin^{N-2} \theta_1 \dtheta_1dr.
\end{align}
From above, due to the lack of the boundary conditions in the
variables $\theta_2, \ldots, \theta_{N-2}$ we get that
$\lambda_1(\gamma)$ is the optimal constant in the weighted
inequality
\begin{align}\label{theta1}
\int_{0}^{\infty}\int_{0}^{\gamma} \int \frac{u_{\theta_1}^2}{r^2}
r^{N-1}\sin^{N-2} \theta_1 \dtheta_1dr \geq \lambda_1(\gamma)
\int_{0}^{\infty}\int_{0}^{\gamma} \int \frac{u^2}{r^2}
r^{N-1}\sin^{N-2} \theta_1 \dtheta_1dr.
\end{align}
More precisely, $\lambda_1(\gamma)$ may be characterized by
\begin{equation}\label{firsteigenvalue}
\lambda_1(\gamma)=\inf_{\{u\in H,  u\neq
 0\}}\frac{\int_{0}^{\gamma}u_{\theta_1}^2\sin^{N-2}\theta_1\dtheta_1}{
\int_{0}^{\gamma}u^2\sin^{N-2}\theta_1\dtheta_1},
\end{equation}
where 
$H$ is the completion of the space $$\{u\in C^\infty[0,\gamma) \ |\
u \textrm{ vanishes in a neighborgood of $\gamma$} \}$$ in the norm
\begin{equation}\label{H1}
||u||_{H}^{2}=\int_{0}^{\gamma}
u_{\theta_1}^{2}\sin^{N-2}\theta_1\dtheta_1.
\end{equation}
 Indeed,
if $a_n=a_n(\theta)$ is an approximating sequence in
\eqref{firsteigenvalue} then the sequence
$$u_n:=u_1(r)a_n(\theta_1)u_2(\theta_2)\ldots u_{N-2}(\theta_{N-2}),$$
where $u_1$ is smooth and vanishes in the neighborhood of
$r=\infty$, minimizes also $\lambda_1(\gamma)$.
\subsection{Proofs of Theorems}
\begin{proof}[Proof of Theorem \ref{taa3}(\ref{dd})]
Without losing the generality,  we are going to consider $u$ as in
\eqref{ineq2}. Next we propose the change of variables
 $$v(r,\theta_1,
\theta_2,\ldots, \theta_{N-2}):=u(r, \theta_1, \theta_2,\ldots,
\theta_{N-2})/\cos(\frac{\pi}{2\gamma}\theta_1).$$
 For simplicity,
we write $u(r)$ or $u(\theta_1)$ when referring  to the radial
variable respectively at the angular part $\theta_1$.  Thus,
integrating by parts we get the following identity:
\begin{align}\label{eqaa1}
\ig u_{\theta_1}^{2}(\theta_1)\sin^{N-2} \theta_1 \dtheta_1&=\ig
v_{\theta_1}^2(\theta_1)
\cos^2(\frac{\pi}{2\gamma}\theta_1)\sin^{N-2} \theta_1
\dtheta_1+\frac{\pi^2}{4\gamma^2}\ig u^2(\theta_1)\sin^{N-2}
\theta_1 \dtheta_1\nonumber\\ &+(N-2)\frac{\pi}{2\gamma}\ig
v^2(\theta_1)\sin (\frac{\pi}{2\gamma}\theta_1)\cos
(\frac{\pi}{2\gamma}\theta_1) \cos \theta_1 \sin^{N-3}\theta_1
\dtheta_1.
\end{align}
Using the identity (\ref{eqaa1}) and the characterization of
$\lambda_1(\gamma)$ stated in  \eqref{firsteigenvalue}, 
it is enough to show that
$$\cos (\frac{\pi}{2\gamma}\theta_1)\sin (\frac{\pi}{2\gamma}\theta_1) \cos
\theta_1\geq
\frac{\pi}{2\gamma}\cos^2(\frac{\pi}{2\gamma}\theta_1)\sin
 \theta_1,\quad \forall \quad \theta_1\in [0,\gamma].$$
Obviously, this is true for $\theta_1=\{0,\gamma\}$. Dividing by
$\frac{\pi}{2\gamma}\theta_1\cos^2(\frac{\pi}{2\gamma}\theta_1)\cos
\theta_1$ it remains to prove that
$$\frac{\tan (\frac{\pi}{2\gamma}\theta_1) }{\frac{\pi}{2\gamma}\theta_1}\geq \frac{\tan \theta_1}{\theta_1}, \quad \forall \quad \theta_1\in (0,\gamma).$$
Because $\pi/2\gamma>1$, the last inequality is true due to the
increasing monotonicity of the function $\theta_1\mapsto \tan
\theta_1/\theta_1$ in the interval $(0,\pi/2)$. This ends the proof.
\end{proof}
\begin{proof}[ Proof of Theorem \ref{taa3}(\ref{cc})]
In the sequel we will use the notations $$\gamma=\pi/2+\delta \quad
\textrm{
 and } \overline{\gamma}:=\pi/2+\delta/2,$$ where $\delta>0$. It suffices to prove that
\begin{prop}\label{taa2}
For any  $\eps>0$, $\eps<<1$, there exists $\delta=\delta(\eps)>0$,
such that
\begin{equation}\label{eqaa6}
\int_{0}^{\overline{\gamma}(\eps)}
u_{\theta_1}^2(\theta_1)\sin^{N-2} \theta_1\dtheta_1 \geq
(N-1-\eps)\int_{0}^{\overline{\gamma}(\eps)}
 u^2(\theta_1)\sin^{N-2} \theta_1\dtheta_1, 
 \end{equation}
 for all $u\in H$. 
\end{prop}
Next we prove Proposition \ref{taa2}.
 Denote
$$v(\theta_1):=u(\theta_1)/\cos(\frac{\pi}{2\gamma}\theta_1),$$
and similar to  \eqref{eqaa1} we obtain
\begin{align}\label{eqaaa1}
\igo u_{\theta_1}^{2}(\theta_1)\sin^{N-2} &\theta_1 \dtheta_1=\igo
v_{\theta_1}^2(\theta_1)
\cos^2(\frac{\pi}{2\gamma}\theta_1)\sin^{N-2} \theta_1
\dtheta_1+\frac{\pi^2}{4\gamma^2}\igo u^2(\theta_1)\sin^{N-2}
\theta_1 \dtheta_1\nonumber\\
 &+(N-2)\frac{\pi}{2\gamma}\igo
v^2(\theta_1)\sin (\frac{\pi}{2\gamma}\theta_1)\cos
(\frac{\pi}{2\gamma}\theta_1) \cos \theta_1 \sin^{N-3}\theta_1
\dtheta_1.
\end{align}
\noindent Then
\begin{align}\label{mainhardy}
\igo u_{\theta_1}^2(\theta_1)\sin^{N-2} \theta_1 \dtheta_1 &\geq
\frac{\pi^2}{4\gamma^2}\igo u^2(\theta_1) \sin^{N-2} \theta_1
\dtheta_1\nonumber\\
&+(N-2)\frac{\pi}{2\gamma}\igo v^2(\theta_1)\sin
(\frac{\pi}{2\gamma}\theta_1)\cos (\frac{\pi}{2\gamma}\theta_1) \cos
\theta_1 \sin^{N-3}\theta_1 \dtheta_1\nonumber\\&:=A+B.
\end{align}
Let $\eps>0$, $\eps<<1$ to be  fixed. In order to compute $B$, we
split it in two parts as
$$B=(N-2)\frac{\pi}{2\gamma}\int_{0}^{\pi/2-\eps}\ldots +
(N-2)\frac{\pi}{2\gamma}\int_{\pi/2-\eps}^{\overline{\gamma}}\ldots:=B_1+B_2.$$
where \noindent Next we concentrate on $B_1$.

Firstly, there exists $\delta=\delta(\eps)>0$, such that (cf. Lemma
\ref{laa3})
\begin{equation}\label{eqaa7}
\sin (\frac{\pi}{2\gamma(\eps)}\theta_1) \cos \theta_1\geq (1-\eps)
\cos (\frac{\pi}{2\gamma(\eps)}\theta_1) \sin \theta_1, \quad
\forall
\quad 0\leq \theta_1\leq \frac{\pi}{2}-\eps. 
\end{equation}
Using this and undoing the variables we reach to
\begin{align}\label{B1}
B_1&\geq (N-2)\frac{\pi}{2\gamma(\eps)}(1-\eps)
\int_{0}^{\pi/2-\eps}u^2(\theta_1)\sin^{N-2}\theta_1\dtheta_1\nonumber\\
&=(N-2)\frac{\pi}{2\gamma(\eps)}(1-\eps)\int_{0}^{\overline{\gamma}(\eps)}
u^2(\theta_1)\sin^{N-2}\theta_1\dtheta_1-B_{11},
\end{align}
where $$B_{11}=(N-2)\frac{\pi}{2\gamma(\eps)}(1-\eps)
\int_{\pi/2-\eps}^{\overline{\gamma}(\eps)}u^2(\theta_1)\sin^{N-2}
\theta_1\dtheta_1,$$ and thus
\begin{align}\label{B11}
|B_{11}|\leq C_1
\int_{\pi/2-\eps}^{\overline{\gamma}(\eps)}u^2(\theta_1)\dtheta_1
\end{align}
for some positive constant $C_1$ independent of $\eps$.
 Undoing the variables we have
\begin{align*}\label{eqau53}
B_2&=(N-2) \frac{\pi}{2\gamma(\eps)}
\int_{\pi/2-\eps}^{\overline{\gamma}(\eps)}u^2(\theta_1)\Big(\frac{\cos
\theta_1}{\cos (\frac{\pi}{2\gamma(\eps)}\theta_1)}\Big)\sin
(\frac{\pi}{2\gamma(\eps)}\theta_1) \sin^{N-3}\theta_1 \dtheta_1.
\end{align*}
Again, there exists $\delta=\delta(\eps)>0$ and a constant $C$
independent of $\eps$, such that (cf. Lemma \ref{laa5})
\begin{equation*}\label{eqaa25}
\Big|\frac{\cos \theta_1}{\cos
(\frac{\pi}{2\gamma(\eps)}\theta_1)}\Big|<C, \quad\forall\quad 0\leq
\theta_1\leq \overline{\gamma}(\eps),
\end{equation*}
and reconsidering the constant $C$ we obtain
\begin{align}\label{B2}
|B_2|\leq
C\int_{\pi/2-\eps}^{\overline{\gamma}(\eps)}u^2(\theta_1)\dtheta_1
\end{align}

From \eqref{B11} and \eqref{B2} we deduce
\begin{equation}\label{error}
|B_{11}|+|B_2|\leq
C\int_{\pi/2-\eps}^{\overline{\gamma}(\eps)}u^2(\theta_1)\dtheta_1
\end{equation}
On the other hand, from  Leibnitz-Newton formula there exists
$\tau=\tau(\eps) \in (\pi/2-\eps,\overline{\gamma}(\eps))$ such that

\begin{equation}\label{eqe4}
\int_{\pi/2-\eps}^{\overline{\gamma}(\eps)}u^2(\theta_1)\dtheta_1
=(\overline{\gamma}(\eps)+\eps-
\pi/2)u^2(\tau(\eps))=(\delta(\eps)/2+\eps)u^2(\tau(\eps)).
\end{equation}
 Applying Holder inequality we find
\begin{align}\label{eqe1}
u^2(\tau(\eps))&=\Big|\int_{\tau(\eps)}^{\overline{\gamma}(\eps)}
u_{\theta_1}(\theta_1)\dtheta_1\Big|^2\leq
C\int_{\tau(\eps)}^{\overline{\gamma}(\eps)}u_{\theta_1}^2(\theta_1)\dtheta_1
\leq C_1\int_{0}^{\overline{\gamma}(\eps)}
u_{\theta_1}^2(\theta_1)\sin^{N-2}\theta_1\dtheta_1,
\end{align}
where $C$, $C_1$ are some positive constants independent of $\eps$.
From \eqref{error}, (\ref{eqe1}) and (\ref{eqe4}) we deduce
\begin{equation}\label{eqe2}
|B_{11}|+|B_2|\leq C(\delta(\eps)+\eps)
\int_{0}^{\overline{\gamma}(\eps)}
u_{\theta_1}^2(\theta_1)\sin^{N-2}\theta_1\dtheta_1.
\end{equation}
Hence, according to \eqref{mainhardy} we get
\begin{align}\label{eqe3}
\big(1-C(\delta(\eps) +\eps)\big)&\int_{0}^{\overline{\gamma}(\eps)}
u_{\theta_1}^2(\theta_1)\sin^{N-2} \theta_1 \dtheta_1
\geq\nonumber\\
& \geq
\Big(\frac{\pi^2}{4\gamma^2(\eps)}+(N-2)\frac{\pi}{2\gamma(\eps)}(1-\eps)\Big)
\int_{0}^{\overline{\gamma}(\eps)} u^2(\theta_1)\sin^{N-2} \theta_1
\dtheta_1.
\end{align}

In other words
\begin{equation}\label{eqe3}
\ig u_{\theta_1}^2(\theta_1)\sin^{N-2} \theta_1 \dtheta_1\geq
\frac{\frac{\pi^2}{4\gamma^2(\eps)}+(N-2)\frac{\pi}{2\gamma(\eps)}(1-
\eps)}{1-C(\delta(\eps) +\eps)} \ig u^2(\theta_1)\sin^{N-2} \theta_1
\dtheta_1.
\end{equation}

Due to the fact that
$$\lim_{\delta\searrow 0}\frac{\frac{\pi^2}{4\gamma^2}+(N-2)\frac{\pi}{2\gamma}(1-\eps)}{1-C(\delta
+\eps)}=\frac{N-1-(N-2)\eps}{1-C\eps},$$ we may reconsider
$\delta=\delta(\eps)>0$ from above such that
$$\frac{\frac{\pi^2}{4\gamma^2(\eps)}+(N-2)\frac{\pi}{2\gamma(\eps)}(1-\eps)}{1-C(\delta(\eps)
+\eps)}> \frac{N-1-(N-2)\eps}{1-C\eps}-\frac{\eps}{2}. 
 $$

Now we can choose $\eps>0$ small enough such that
$$\frac{N-1-(N-2)\eps}{1-C\eps}\geq N-1-\frac{\eps}{2},$$
and for each $\delta=\delta(\eps)$ from above, we have
$$\frac{\frac{\pi^2}{4\gamma^2(\eps)}+(N-2)\frac{\pi}{2\gamma(\eps)}(1-\eps)}{1-C(\delta(\eps)
+\eps)}=\frac{N-1}{1-C\eps}>N-1-\eps.
$$ We end up the proof of the second part of Theorem \ref{taa3}.
\end{proof}

\begin{proof}[Proof of Theorem \ref{SL1}]
with $\gamma\in (0, \pi)$. Next, the aim is to show upper and lower
bounds for the value  $\lambda_1(\gamma)$ in
\eqref{firsteigenvalue}. Due to inequalities
 \be\label{equivalence1}
\frac{\sin \gamma}{\gamma }t \leq \sin t \leq t , \quad \forall
\quad t\in (0, \gamma), \gamma\in (0,\pi), \ee it suffices to
determine the value of
 \begin{equation}\label{firsteigenvalue2}
\lambda_{\star}^1(\gamma):=\inf_{\{u\in H, u\neq
 0\}}\frac{\int_{0}^{\gamma}u_{t}^2t^{N-2}dt}{
\int_{0}^{\gamma}u^2 t^{N-2}dt},
\end{equation}
which is well defined in $H$ since
$$\int_{0}^{\gamma}u_t^2 t^{N-2}dt
< \infty \textrm{ iff } \int_{0}^{\gamma}u_t\sin^{N-2} t dt< \infty,
$$
$H$ being defined by the norm in \eqref{H1}.
 By Proposition \ref{prop1}, it holds
that $\lambda_{\star}^{1}(\gamma)>0$. Due to the compact embedding
(see Proposition \ref{prop2})
$$H \hookrightarrow\loogt,$$
  $\lambda_{\star}^{1}(\gamma)$ is attained by a non-trivial function
  $\phi_1$. Then one can prove that $\lambda_1^{\star}(\gamma)$
  satisfies the variational problem: there exists $\phi^1\in H$ such
  that
  \begin{equation}\label{variational}
  \int_{0}^{\gamma}\phi_{t}^1v_t
  t^{N-2}dt=\lambda_1^{\star}(\gamma)\int_{0}^{\gamma}\phi^1v t^{N-2}dt,
  \quad \forall \quad v\in H.
  \end{equation}
Next we note that any $u\in H$ exhibits a hidden  weak Neumann
boundary condition at the origin $t=0$:
\begin{equation}
\lim_{\eps \rightarrow 0}\frac{1}{\eps}\int_{0}^{\eps}
u_tt^{N-2}dt=0.
\end{equation}
Indeed, we have
\begin{align}\label{BC1}
\Big|\frac{1}{\eps}\int_{0}^{\eps} u_tt^{N-2}dt\Big|&\leq
\frac{1}{\eps}\Big(\int_{0}^{\eps}(u_t)^2t^{N-2}dt\Big)^{1/2}
\Big(\int_{0}^{\eps}t^{N-2}dt\Big)^{1/2}=\eps^{(N-3)/2}\Big(\int_{0}^{\eps}
u_tt^{N-2}dt\Big)^{1/2}
\end{align}
which converges to 0 when $\eps$ tends to 0. This allows to make
integrations by parts and rewrite \eqref{variational} as
\begin{equation}\label{BC2}
\int_{0}^{\gamma}-(\phi_t^1t^{N-2})_t v
dt=\lambda_{1}^{\star}(\gamma) \int_{0}^{\gamma}\phi^1 vt^{N-2}dt,
\quad \forall\quad v\in H.
\end{equation}

 Therefore,
$\lambda_{\star}^{1}(\gamma)$ is the first  eigenvalue of the
degenerate Sturm-Liouville problem

\begin{equation}\label{SL2}
\left \{\begin{array}{ll}
  -(u_t t^{N-2})_t =\lambda  u t^{N-2},  &  t\in (0, \gamma), \\
 \lim_{t\rightarrow 0}u_tt^{N-2}=0,\quad  u(\gamma)=0, &  \\
\end{array}\right.
\end{equation}
 with the corresponding
eigenvector $\phi_1$.

In the sequel we  determine explicitly the value of
$\lambda_{\star}^{1}(\gamma)$.

With the change of variables $v=u t^{N-2}$, the problem \eqref{SL2}
reduces to the following Bessel equation with boundary constraint
\begin{equation}\label{SL3}
\left \{\begin{array}{ll}
   v_{tt}+(2-N)\frac{v_t}{t}+\Big(\lambda +\frac{N-2}{t^2}\Big)v=0,  &  t\in (0, \gamma), \\
  v(\gamma)=0, &  \\
\end{array}\right.
\end{equation}
\subsubsection{Bessel functions}
If $n$ is positive integer then, the first Bessel function $J_n$ of
order $n$ has the expression
\begin{align}
J_n(x)=\frac{x^n}{2^n n!}\Big(1&-\frac{x^2}{2\cdot(2n+2)}
+\frac{x^4}{2\cdot 4\cdot (2n+2)\cdot (2n+4)}-\ldots\Big)
\end{align}
and $J_n$ behaves like $x^n$ when $x>0$ is small. If $n$ is a
negative integer, by definition yields
$$J_{-n}(x)=(-1)^n J_n(x).$$
If $n$ is not an integer then
\begin{align*}
J_n(x)=\frac{x^n}{2^n \Gamma(n+1)}\Big(1&-\frac{x^2}{2\cdot(2n+2)}
+\frac{x^4}{2\cdot 4\cdot (2n+2)\cdot (2n+4)}-\ldots\Big),
\end{align*}
where $\Gamma$ denotes the Gamma-function. When $n$ is an integer it
is necessary to recall the so-called Weber's function, i.e.
$$Y_n(x)=J_n(x)\int \frac{dx}{xJ_n^2(x)},$$
   which behaves  like $1/x^n$ when
$x>0$ is small.
 Next, we consider the Bessel equation
\begin{equation}\label{SL4}
y_{tt}-(2\alpha-1)\frac{y_t}{t}+\Big(\beta^2 \tau^2 t^{2\tau
-2}+\frac{\alpha^2 -n^2\tau^2}{t^2}\Big)y=0,
\end{equation}
Due to  \cite{bessel}, pp. 117,  the general solution of \eqref{SL4}
is given by
$$y=t^{\alpha}\{ A J_{n}(\beta t^{\tau})+B Y_{n}(\beta t^{\tau})\},$$
$$y=t^{\alpha}\{ A J_{n}(\beta t^{\tau})+B J_{-n}(\beta t^{\tau})\},$$
where $A, B$ are constants, 
according as $n$ is non-negative  integer or not.

Once $\lambda\neq 0$ is an eigenvalue for \eqref{SL2} then $\lambda$
is also an eigenvalue  in \eqref{SL3}. The general solution of
\eqref{SL3} is a particular case of \eqref{SL4} for
$\alpha=(N-1)/2$, $\tau=1$, $\beta=\sqrt{\lambda}$, $n=(N-3)/2$,
i.e.,
$$v(t)=t^{\frac{N-1}{2}}\{AJ_{\frac{N-3}{2}}(\sqrt{\lambda}t)+
BY_{\frac{N-3}{2}}(\sqrt{\lambda}t)\},$$ or
$$v(t)=t^{\frac{N-1}{2}}\{AJ_{\frac{N-3}{2}}(\sqrt{\lambda}t)+
BJ_{-\frac{N-3}{2}}(\sqrt{\lambda}t)\}.$$

We show that it must be $B=0$, in which  case it simplifies the
expression of $v$ i.e.
\begin{equation}\label{SL5}
v(t)=At^{\frac{N-1}{2}}J_{\frac{N-3}{2}}(\sqrt{\lambda}t).
\end{equation}

Indeed, if $N=3$ then, reconsidering the constants, it is trivial
that $v$ is as in \eqref{SL5}. Assume  $N\geq 4$ and $B\neq 0$.
Using the behavior of $J_n$ and $Y_n$ at zero we get that
$$v(t)\sim AC_1t^{N-1}+BC_2t,$$
where $C_1=C_1(\lambda)$, $C_2=C_2(\lambda)$ are non-trivial
constants depending on $\lambda$. Consequently, up to a constant,
$$u(t)\sim 1+ Ct^{3-N}$$
with $C\neq 0$. Since $N\geq 4$, this last formula yields to
$$\int_{0}^{\gamma} u^2(t) dt=\infty,$$
which contradicts the fact that $u\in H$. Hence, the assumption is
false and $B=0$ for any $N\geq 4$.

Imposing the condition $v(\gamma)=0$ in the simplified expresion
\eqref{SL5},  we obtain $\sqrt{\lambda}=B_n$, where $\{B_n\}_{n}$
are the positive zero's of the Bessel function $J_{\frac{N-3}{2}}$.
In particular we obtain
$$\lambda_{\star}^1(\gamma)=B_1^2/\gamma^2.$$ 
Using this, the relations \eqref{firsteigenvalue},
\eqref{firsteigenvalue2} and  the inequality \eqref{equivalence1}
 we obtain the
conclusion of Theorem \ref{SL1}.
\end{proof}
\subsection{ Proofs of useful results}
\begin{prop}\label{prop1}
For any $ v\in H$ we have
 \be\label{keypoint} \intog v_t^2
t^{N-2} dt \geq \frac{1}{\gamma}\Big(\frac{N-2}{2}\Big)^2 \intog v^2
t^{N-3} dt. \ee
\end{prop}
\begin{proof}[Proof of Proposition \eqref{prop1}]
Of course, we have
$$\intog v_t^2t^{N-2} dt\geq \frac{1}{\gamma}\intog v_t^2 t^{N-1}dt,$$
and applying the Hardy inequality in $N$-d  we complete the proof of
Proposition \eqref{prop1}.
\end{proof}

\begin{prop}\label{prop2} The embedding
$$H \hookrightarrow\loogt$$
is compact.
\end{prop}
\begin{proof}[Proof of Proposition \ref{prop2}]
The key point is 
played by  Proposition \ref{prop1}.

Next we consider a sequence $(u_n)_n\rightharpoonup 0$ in $H$ and
suffices to prove its strong convergence in $\loogt$ i.e.
$u_n\rightarrow 0$ in $\loogt$. By weak convergence, $\{u_n\}_n$ is
bounded in $\hogt$, let's say by
a constant $C$. Accordingly to Proposition \ref{prop1} 
  we have as well
\be\label{anotherkey} 
 \quad \intog u_n^2 t^{N-3} dt \leq C, \quad \forall n\in \nn.
 \ee
Given $\eps>0$ we split the $L^2$-norm by distinguish that
concentrated in $B(0,\eps)$ and in its exterior:
\begin{equation}\label{eq95}
||u_n||_{\loogt}^{2}=\int_{0}^{2\eps}|u_n|^2 t^{N-2}
dt+\int_{2\eps}^{\gamma}|u_n|^2 t^{N-2} dt:=I_{\eps, n}^{1}+I_{\eps,
n}^{2}.
\end{equation}
Let us also consider the partition of unity of $u_n$,
$$u_n=u_n\varphi+(1-\varphi)u_n:=w_{1, n}+w_{2, n},$$
where $\varphi$ is a regular function such that
\begin{equation}
\varphi(t)=\left\{\begin{array}{ll}
  1, & t\leq \eps, \\
  0, & t\geq  2\eps, \\
\end{array}\right.
\end{equation}
  Obviously,
$\textrm{supp}(w_{1,n})\subset (0, 2\eps)$, $\textrm{supp}(w_{2,
n})\subset (\eps, \gamma)$. Firstly, from \eqref{anotherkey} we have
\begin{align}\label{eq96}
I_{\eps, n}^{1}&\leq 2\eps \int_{0}^{2\eps} |u_n|^2 t^{N-3}dt \leq
2\eps \intog |u_n|^2t^{N-3}
dt\nonumber\\
\end{align}
Secondly, let us notice that
\begin{equation}\label{eq98}
w_{2, n} \rightharpoonup 0 \textrm { in } \hogt.
\end{equation}
For this, it suffices   to prove $(\psi, w_{2,n})_{\hogt}\rightarrow
0$ for all $\psi\in C_{c}^{\infty}$. We evaluate,
\begin{align*}\label{eq97}
(\psi,w_{2, n})_{\hogt}&=\intog  \psi_t((1-\varphi)u_n)_t t^{N-2}
dt=-\intog (\psi_t t^{N-2})_t (1-\varphi) u_n  dt
\end{align*}
which converges to 0 when $n\rightarrow \infty$. This happens
because weak convergence  in  $H$ involves  weak convergence in
$\loogdt$ (by Proposition \eqref{prop1}). Now we observe that the
support of $w_{2, n}$ lies far from zero and therefore the norm of
$w_{2, n}$ in $H$ is equivalent to the norm of $w_{2, n}$ in
$H_{0}^{1}(\eps, \gamma)$. 
But
$H_{0}^{1}(\eps, \gamma)$ is compact embedded in $L^2(\eps, \gamma)$, and in particular  in $\loogdt$. 
 We obtain that $w_{2,n}
\rightarrow 0 \textrm{ in } \loogdt$.
Hence, we can choose $n$ large enough such that $I_{\eps,
n}^{2}<\eps$. From here and (\ref{eq96}), we conclude that $u_n$
converges strongly to 0 in $\loogt$.
\end{proof}

\begin{lema}\label{laa4}
Let us consider $a<1$. Then, the application
$$(0,\frac{\pi}{2})\backepsilon t\rightarrow \frac{\tan at}{\tan t} \in (0,\infty)$$
is decreasing.
\end{lema}
\begin{proof}[Proof of Lemma \ref{laa4}]
Indeed, if we consider $f(t)=\frac{\tan at}{\tan t}$ we obtain
$$f'(t)=\frac{at}{\cos ^2 at \sin^2 t}\Big(\frac{\sin 2t}{2t}-\frac{\sin 2at}{2at}\Big).$$
It follows that $f'<0$ due to the decreasing behavior of the
function $x\mapsto
\frac{\sin x}{x}$ on $(0,\pi)$. 

With this, we complete the proof.
\end{proof}

\begin{lema}\label{laa3}
For any $\eps>0$, $\eps <<1$, there exists $\delta=\delta(\eps)>0$
such that
\begin{equation}\label{eqaa7}
\sin (\frac{\pi}{2\gamma(\eps)}\theta_1) \cos \theta_1\geq (1-\eps)
\cos (\frac{\pi}{2\gamma(\eps)}\theta_1) \sin \theta_1, \quad
\forall
\quad 0\leq \theta_1\leq \frac{\pi}{2}-\eps. 
\end{equation}
\end{lema}

\begin{proof}[Proof of Lemma \ref{laa3}]
 Let us put formally
$\gamma=\pi/2+\delta$ with $\delta>0$. Then (\ref{eqaa7}) becomes
\begin{equation*}\label{eqa8}
\sin (\frac{\pi}{\pi+2\delta(\eps)}\theta_1)\cos \theta_1\geq (1
-\eps)\cos (\frac{\pi}{\pi+2\delta(\eps)}\theta_1)\sin \theta_1,
\end{equation*}
or
\begin{equation*}\label{eqa9}
\tan (\frac{\pi}{\pi+2\delta(\eps)}\theta_1)\geq (1-\eps)\tan
\theta_1,\quad \forall\quad  0\leq \theta_1 \leq\pi/2-\eps,
\end{equation*}
or equivalent to
\begin{equation}\label{eqa12}
\frac{\tan (\frac{\pi}{\pi+2\delta(\eps)}\theta_1)}{\tan
\theta_1}\geq (1-\eps),\quad \forall \quad 0\leq \theta_1 \leq
\pi/2-\eps.
\end{equation}
Now, let us show the validity of (\ref{eqa12}). Because
$$\lim_{\delta\searrow 0}\frac{\tan\Big(\frac{\pi}{\pi+2\delta}\big(\frac{\pi}{2}-\eps\big)\Big)}{\tan
(\frac{\pi}{2}-\eps)}=1,$$
 we can choose $\delta=\delta(\eps)>0$ such that
$$\frac{\tan\Big(\frac{\pi}{\pi+2\delta(\eps)}\big(\frac{\pi}{2}-\eps\big)\Big)}{\tan (\frac{\pi}{2}-\eps)}\geq (1-\eps).$$
By this inequality and Lemma \ref{laa4} it is easy to obtain
(\ref{eqa12}).
\end{proof}

\begin{lema}\label{laa5}
Let $0<\eps<<1$. There exists $\delta=\delta(\eps)>0$ and a constant
$C$, such that
\begin{equation}\label{eqaa25}
\Big|\frac{\cos \theta_1}{\cos
(\frac{\pi}{2\gamma(\eps)}\theta_1)}\Big|<C, \quad\forall\quad 0\leq
\theta_1\leq \overline{\gamma}(\eps),
\end{equation}
\end{lema}

\begin{proof}[Proof of Lemma \ref{laa5}]
Let us put $\theta_1=\pi/2+\delta (\eps )t''$ with $t''\in (0,1/2)$.
Then
\begin{eqnarray}\label{la6}
\Big|\frac{\cos \theta_1}{\cos
(\frac{\pi}{2\gamma(\eps)}\theta_1)}\Big|&=
&\Big|\frac{\sin\delta(\eps)t''}{\sin\Big(\frac{\pi\delta(\eps)}
{\pi+2\delta(\eps)}(1-t'')\Big)}\Big| \rightarrow \frac{t''}{1-t''},
\end{eqnarray}
 when $\delta(\eps)\rightarrow 0,\quad  \forall\quad  t''\in (0,1/2).$ But
$$\sup_{t''\in (0,1/2)}\frac{t''}{1-t''}=1,$$ and this yields to the
conclusion of Lemma.
\end{proof}

\section{Further comments and open problems}\label{sec4}
\subsection{Efficiency of the methods and sharp reminder terms} As we mentioned in Theorem \ref{tu2}, the inequality we
obtained involves a reminder term of order $\int v^2/|x|dx$ in the
lower bound. The proof uses a change of variable adapted to the
boundary near the singular point. Comparing Theorems \ref{tu2} and
\ref{tu3}, we see  that the results improve when using spherical
harmonics decomposition. More precisely, the inequality stated in
Theorem  \ref{tu3} admits an optimal reminder term of order $\int
v^2/(|x|^2\log^2(1/|x|))dx$. Thus, spherical harmonics decomposition
yields better results.

\subsection{Inequalities in cones}
In 2-d we have given  a complete picture of the sharp Hardy
inequality. In the multi-dimensional case, $N\geq 3$, we proved
several qualitative inequalities but optimality results are still to
be proved. To be more precise, in convex cones we  proved that the
constant is at least $(N-2)^2/4+(N-1)\pi^2/4\gamma^2$,  $\gamma$
being the slot of the cone (see Subsection \ref{taa3}). This result
shows that the best constant  blows-up when $\gamma$ tends to 0. For
concave cones we have obtained less information: we have only shown
that the constant varies continuously with respect to the slot
$\gamma$ nearby $\gamma=\pi/2$. For any cone with the slot
$\gamma\in (0, \pi)$ we have proved that the best constant is at
least $(\sin \gamma/ \gamma)^{N-2}B_1^2/\gamma^2$, where $x_1$ is
the first positive zero of the Bessel function $J_{(N-3)/2}$.  To
our knowledge, explicit formulas for the optimal constant are still
to be proved.

\subsection{Weak Hardy inequalities with $L^2$-reminder terms}

In the context of smooth domains, for certain geometries,  we have
improved the Hardy
 constant
 from $ (N-2)^2/4$ to $N^4/4$. To do
this,  we had to add a  $L^2$-reminder term in the upper bound of
the inequality. This extra term in $L^2$-norm cannot be disregarded
as shown in Proposition \ref{hardysub}. Thus,  the inequalities that
we obtain are sharp. However, the problem on the possible existence
of domains  $\Omega$ satisfying C$\ref{eqq3}$ such that
$\mu(\Omega)=N^2/4$ is open. In the case where $\Omega$ has an
hyperbolic geometry at the origin, as $\eps$ tends to 0, the
constant $C(\Omega, \eps)$ is expected to blow-up (see Theorem
\ref{ta5}). This class of generalized Hardy inequalities with lower
order reminder terms is of application in various contexts.  For
instance, the Hardy inequalities play a crucial role when studying
the controllability of wave equations with quadratic singular
potentials. In that setting, one can get rid of the $L^2$-reminder
terms, using compactness-uniqueness arguments (see \cite{cristi1}).

\vspace{0.5cm}

\noindent {\bf Acknowledgements.} I am really grateful to Enrique
Zuazua for his guidance.  I wish to thank Adi Adimurthi, David
Krejcirik, Kyril Tintarev, Liviu Ignat, Alessio Porretta, for
fruitful discussions.


\end{document}